\documentclass[12pt,reqno]{amsart}
\usepackage{fullpage}
\usepackage{colonequals}
\usepackage{amsmath,amssymb,amsthm,url}
\usepackage{xcolor}
\usepackage[utf8]{inputenc}
\usepackage[english]{babel}
\usepackage{bbm}
\usepackage{enumerate}
\usepackage{tikz}
\usepackage{enumitem}
\usetikzlibrary{calc}
\usetikzlibrary{positioning}
\usepackage{graphicx}
\usepackage{mathrsfs}
\usepackage[colorlinks=true, pdfstartview=FitH, linkcolor=blue, citecolor=blue, urlcolor=blue]{hyperref}

\newtheorem{thm}{Theorem}[section]
\newtheorem{lem}[thm]{Lemma}
\newtheorem{prop}[thm]{Proposition}
\newtheorem{cor}[thm]{Corollary}

\theoremstyle{definition}
\newtheorem{defn}[thm]{Definition}

\newcommand{\R}{\mathbb{R}} 
\newcommand{\C}{\mathcal{C}} 

\newcommand{\Z}{\mathbb{Z}}
\newcommand{\diam}{\operatorname{diam}}
\newcommand{\dist}{\operatorname{dist}}

\title{Extremal diameters of 3-coloring graphs of trees}

\author{Shamil Asgarli}
\address{Department of Mathematics \& Computer Science \\ Santa Clara University \\ CA 95053 \\ USA}
\email{sasgarli@scu.edu}

\author{Sara Krehbiel}
\address{Department of Mathematics \& Computer Science \\ Santa Clara University \\ CA 95053 \\ USA}
\email{skrehbiel@scu.edu}

\author{Simon MacLean}
\address{Department of Mathematics \& Computer Science \\ Santa Clara University \\ CA 95053 \\ USA}
\email{smaclean@scu.edu}

\author{Gjergji Zaimi}
\email{gjergjiz@gmail.com}



\begin{document}

\begin{abstract}
Given a tree $T$, its 3-coloring graph $\mathcal{C}_3(T)$ has as vertices the proper 3-colorings of $T$, with edges joining colorings that differ at exactly one vertex. We call the diameter of $\mathcal{C}_3(T)$ the \emph{3-coloring diameter} of $T$. We introduce the notion of \emph{balanced labelings} of $T$ and show that the 3-coloring diameter equals the maximum $L_1$-norm of a balanced labeling. Using this equivalence, we determine the maximum and minimum values of the 3-coloring diameter over all trees on $n$ vertices and characterize the extremal trees.
\end{abstract}

\maketitle

\section{Introduction}

Given a graph $G=(V,E)$, a \emph{proper $k$-coloring} is an assignment $c\colon V\to \{1, 2, \dots, k\}$ such that $c(u)\neq c(v)$ for every edge $uv\in E$. The number of proper $k$-colorings of $G$ is the \emph{chromatic polynomial} $\pi_G(k)$, which is a polynomial function of $k$. Another natural object to consider is the $k$-coloring graph $\C_k(G)$, which organizes proper $k$-colorings of $G$ as follows. Vertices of $\C_k(G)$ are proper $k$-colorings of $G$, and two vertices $c_1$ and $c_2$ are adjacent in $\C_k(G)$ if they differ on the color of exactly one vertex. Note that the chromatic polynomial $\pi_G(k)$ is the number of vertices in $\C_k(G)$. We view $\C_k(G)$ as a \emph{reconfiguration graph}, where we toggle the color of exactly one vertex to move to a new configuration (see the survey \cite{Heu13} for a broader perspective on combinatorial reconfiguration problems).

In this paper, we determine the maximum and minimum values of the function $T\mapsto \diam(\C_3(T))$ as $T$ ranges over all trees on $n$ vertices. We obtain a complete classification of the trees that achieve these extremal values.

\begin{thm}\label{thm:main} Let $n\geq 7$ be a fixed positive integer. Among all trees $T$ on $n$ vertices:
\begin{enumerate}
    \item The unique tree with maximum 3-coloring diameter is the path $P_n$, and this maximum value is $\binom{\lceil\frac{n}{2}\rceil -1}{2}+\binom{\lfloor\frac{n}{2}\rfloor+2}{2}$.
    \item The unique trees with minimum 3-coloring diameter are the star $K_{1,n-1}$ and the double stars $S(a,b)$ satisfying $|a-b|\le 4$; in each of these cases, the minimum value is $\left\lfloor \tfrac{3n}{2}\right\rfloor$.  
\end{enumerate}
\end{thm}

The \emph{double star} $S(a, b)$ is the tree formed by joining the centers of two stars $K_{1, a}$ and $K_{1, b}$ with an edge. It has $n=a+b+2$ vertices and diameter 3, assuming $a, b\geq 1$. For instance, the double star graphs $S(7, 2)$ and $S(6, 3)$ both have $11$ vertices:

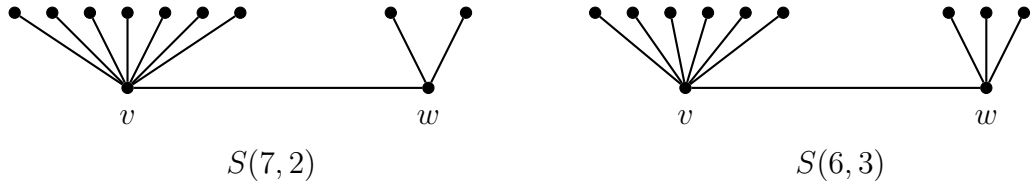
\begin{center}
\begin{minipage}{0.4\textwidth}
\centering
\begin{tikzpicture}[
    vertex/.style={circle, draw, fill=black, inner sep=1.5pt},
    edge/.style={thick},
    label distance=2pt
]

\node[vertex, label=below:$v$] (v) at (0, 0) {};
\node[vertex, label=below:$w$] (w) at (4, 0) {};

\draw[edge] (v) -- (w);

\foreach \i [count=\k from -3] in {1,...,7} {
    \node[vertex] (vl\i) at ({0 + \k * 0.5}, 1) {};
    \draw[edge] (v) -- (vl\i);
}

\node[vertex] (wr1) at (3.5, 1) {}; 
\node[vertex] (wr2) at (4.5, 1) {}; 
\draw[edge] (w) -- (wr1);
\draw[edge] (w) -- (wr2);

\end{tikzpicture}

\vspace{0.03cm} 
$\quad\quad S(7, 2)$
\end{minipage}
\begin{minipage}{0.5\textwidth}
\centering
\begin{tikzpicture}[
    vertex/.style={circle, draw, fill=black, inner sep=1.5pt},
    edge/.style={thick},
    label distance=2pt
]

\node[vertex, label=below:$v$] (v) at (0, 0) {};
\node[vertex, label=below:$w$] (w) at (4, 0) {};

\draw[edge] (v) -- (w);

\foreach \i [count=\k from -2] in {0,...,5} {
    \node[vertex] (vl\i) at ({-0.2+ \k * 0.5}, 1) {};
    \draw[edge] (v) -- (vl\i);
}

\foreach \i [count=\k from -1] in {1,...,3} {
    \node[vertex] (wr\i) at ({4 + \k * 0.5}, 1) {};
    \draw[edge] (w) -- (wr\i);
}

\end{tikzpicture}

\vspace{0.03cm} 
$\quad\quad S(6, 3)$
\end{minipage}
\end{center}

The graph $\C_3(T)$ is isomorphic to the $1$-skeleton of the Hom complex $\operatorname{Hom}(T,K_3)$ \cite{Koz08}. Much of the existing work on these complexes emphasizes their topological properties and applications to graph coloring. In contrast, we focus on the metric structure of these $1$-skeletons, and in particular on $\diam(\C_3(T))$. In the next paragraphs, we place Theorem~\ref{thm:main} in context and explain how it sharpens existing bounds in the literature.

A central open problem in this area is Cereceda's conjecture \cite{Cer07}: if an $n$-vertex graph $G$ is $d$-degenerate, then the diameter of $\C_k(G)$ is at most $O(n^2)$ for $k\geq d+2$. Cereceda proved a weaker version of the conjecture with more colors allowed, namely when $k\ge 2d+1$ \cite[Theorem~5.22]{Cer07}. Even though the stated conclusion of \cite[Theorem~5.22]{Cer07} is $O(n^2)$, the proof yields a more explicit upper bound $\frac{n(n+1)}{2} = \bigl(1+o(1)\bigr)\frac{n^2}{2}$. As a result, Cereceda's conjecture is true in the special case where $d=1$ and $k=3$. In particular, since trees are $1$-degenerate, $\C_3(T)$ has diameter at most $\bigl(1+o(1)\bigr)\frac{n^2}{2}$ for every tree $T$ on $n$ vertices. While the $O(n^2)$ conjecture remains open for general $d$, a polynomial upper bound $O(n^{d+1})$ has been established \cite{BH22}.

On the other hand, if one more color is available, the diameter grows (at most) linearly: Bousquet and Perarnau~\cite{BP16} proved that for every $d$-degenerate graph $G$ on $n$ vertices and every $k\ge 2d+2$, the recoloring diameter satisfies $\diam(\C_k(G))\le (d+1)n$. In particular, $\diam(\C_4(T))\leq 2n$ for every tree $T$ on $n$ vertices. This linear bound helps motivate the focus of the present paper on 3-colorings of trees; some trees genuinely exhibit quadratic growth in the diameter of $\C_3(T)$.

We note that quadratic lower bounds are already known for paths. Bonamy et al. \cite[Theorem~11]{BJLPP14} showed that the 3-coloring diameter of the path $P_n$ is $\Theta(n^2)$, and their proof yields the explicit lower bound $\diam(\C_3(P_n)) \ge \frac{1}{4}(n^2-1)$. This lower bound is consistent with (and implied by) our exact formula for $\diam(\C_3(P_n))$ in Theorem~\ref{thm:main} (1).

We also remark that Cambie, van Batenburg, and Cranston \cite{CvBC24} determined the diameter of $\C_k(K_{p,q})$ for any complete bipartite graph $K_{p,q}$ up to a small additive constant; their result for the special case $\C_{3}(K_{1,n-1})$ is consistent with our findings. The present paper is the first to systematically classify the extremal 3-coloring diameters for trees in general. It would be interesting to further understand the distribution of these values between the extremes. In particular, is there a local move on trees that monotonically increases the 3-coloring diameter? The existence of such a move could yield a poset structure on trees reminiscent of the one from \cite{CL15}. Such a poset would have the path graph as the unique maximal element and the double stars from Theorem~\ref{thm:main} as minimal elements.

\textbf{Structure of the paper.} The paper is organized as follows. In Section~\ref{sec:balanced-labelings}, we introduce the framework of balanced labelings and establish the connection between these labelings and distances in $\C_3(T)$ (Theorem~\ref{thm:labelings} and Corollary~\ref{cor:balanced}). In Section~\ref{sec:median}, we study the median label of a maximum balanced labeling. In Section~\ref{sec:maximum-3-coloring-diameter}, we prove that the path graph $P_n$ uniquely maximizes the 3-coloring diameter for $n\geq 7$. In Section~\ref{sec:minimum-3-coloring-diameter}, we characterize the trees achieving the minimum 3-coloring diameter as the star and nearly symmetric double stars for $n\geq 7$. Finally, Appendix~\ref{sec:small-trees} provides data for trees with $n\le 6$ vertices.

\section{Computing 3-coloring diameter via balanced labelings}\label{sec:balanced-labelings}

Throughout this section, $T$ denotes a tree with vertex set $V$ and edge set $E$. We identify proper 3-colorings of $T$ with functions $f\colon V\to \mathbb{Z}/3\mathbb{Z}$ which satisfy
$f(u)\neq f(v)$ for all $uv\in E$. If $g$ is obtained from $f$ by changing the color at a single vertex $v$ such that $g(v)=f(v)+1$ (respectively $g(v)=f(v)-1$), we say $g$ is obtained from $f$ by a \emph{$(+1)$-toggle} (respectively \emph{$(-1)$-toggle}) at $v$. The \emph{3-coloring graph} $\C_3(T)$ has vertex set equal to all proper 3-colorings of $T$, with edges joining proper colorings that differ by a single vertex toggle (i.e., a $(+1)$ or $(-1)$-toggle).

A \emph{labeling} of $T$ is a map $h\colon V\to \mathbb{Z}$ satisfying $|h(u)-h(v)|\leq 1$ for each edge $uv\in E$. Equivalently, $h$ is a graph homomorphism from $T$ to the infinite path graph on $\mathbb{Z}$ with self-loops at each vertex. In particular, the labels are supported on a contiguous subset of $\mathbb Z$; that is, the image of $h$ is an interval. 

We relate these integer labelings to the reconfiguration graph by viewing them as lifts of the difference between two colorings. Suppose $f$ and $g$ are proper 3-colorings of $T$. A labeling $h\colon V\to \mathbb{Z}$ is \emph{$(f,g)$-compatible} if 
\[
h(v) \equiv g(v)-f(v)\pmod{3}
\]
for every vertex $v\in V$.

We denote the set of all $(f,g)$-compatible labelings by $D(f,g)$. Notice that if $h_1,h_2 \in D(f, g)$, then $h_1(v)-h_2(v)=0 \pmod{3}$ for every $v\in V$.

\begin{lem}\label{lem:shift-3m}
If $h\in D(f,g)$ and $m\in\mathbb{Z}$, then $h+3m\in D(f,g)$. Moreover, every element of $D(f,g)$ has this form for some $m\in\mathbb{Z}$.
\end{lem}

\begin{proof}
First, the map $h'\colonequals h+3m$ satisfies both the congruence condition modulo $3$ and the inequality $|h'(u)-h'(v)|\leq 1$ for all edges $uv\in E$; hence $h'$ is $(f,g)$-compatible. Conversely, suppose that $h'\in D(f,g)$ is arbitrary. Then we have $h'(v)-h(v)\equiv 0\pmod{3}$ for all $v\in V$. For any $uv\in E$, we have
$$|(h'(u)-h(u))-(h'(v)-h(v))|\le |h'(u)-h'(v)|+|h(u)-h(v)|\le 2.$$
Since $(h'(u)-h(u))-(h'(v)-h(v))$ is a multiple of $3$ and has absolute value at most $2$, it equals $0$. We conclude that $h'(u)-h(u)=h'(v)-h(v)$. As $T$ is connected, it follows that $h'(v)-h(v)$ is constant throughout $V$. Thus, $h'(v)-h(v)=3m$ for some $m\in\mathbb{Z}$.
\end{proof}

\begin{lem}\label{lem:homs-representable}
For every labeling $h\colon V\to \mathbb{Z}$, there exist 3-colorings $f,g$ of $T$ such that $h\in D(f,g)$.
\end{lem}

\begin{proof}
We use induction on the number of vertices. When $|V|=1$, the tree has only one vertex $v$; given $h(v)$, choose $f(v)$ arbitrarily and set $g(v)=f(v)+h(v)$.

For the inductive step, let $T$ be a tree on $n\geq 2$ vertices. Suppose that $v$ is a leaf vertex of $T$. By the inductive hypothesis, there exist two 3-colorings $f',g'$ of $T'\colonequals T-v$ such that the restriction of $h$ to $V(T')$ is a labeling in $D(f',g')$. To extend $f'$ to a coloring $f$ of $T$, we have two options for $f(v)$, which can be written as $x,x+1 \pmod{3}$. Similarly $g'$ can be extended to a coloring $g$ of $T$ for some $y,y+1 \pmod{3}$. Since the differences $(y+1)-x, y-x, y-(x+1)$ cover all residues (mod $3$), we choose a pair satisfying $g(v)-f(v)\equiv h(v)\pmod{3}$.
\end{proof}

By Lemma~\ref{lem:shift-3m}, the set $D(f,g)$ is a lattice of labelings separated by uniform shifts of $3$. We are interested in the labelings within this set that minimize the $L_1$-norm.

\begin{defn}
A labeling $h\colon V\to \mathbb{Z}$ is \emph{balanced} if its $L_1$-norm $\|h\|_1 \colonequals \sum_{v\in V}|h(v)|$ satisfies
\[
\|h\|_1\leq \sum_{v\in V}|h(v)+3m|
\]
for all integers $m$.
\end{defn} 

Since the function $x\mapsto \sum_{v\in V}|h(v)+x|$ is convex, this is equivalent to 
\begin{equation}\label{eq:convex-balanced}
\|h\|_1\le \min \left(\|h-3\|_1, \|h+3\|_1\right). 
\end{equation}
In other words, for a labeling $h\colon V\to\Z$ and any pair $(f,g)$ with $h\in D(f,g)$, we have
\[
h \text{ satisfies~\eqref{eq:convex-balanced}} \Longleftrightarrow h \text{ is balanced} \Longleftrightarrow h \text{ minimizes the } L_1 \text{-norm over } D(f, g).
\]
One can think of the tree as being drawn on the integer line, with each vertex
placed at its label.  The balanced condition says that if we slide the entire
arrangement to the left or right by $3$ units, the total ``distance from $0$'' never decreases.

To relate balanced labelings to graph distances in $\C_3(T)$, we first analyze how shortest walks in $\C_3(T)$ behave. The next two lemmas describe structural constraints on such walks.

Suppose $f$ and $g$ are adjacent colorings in $\C_3(T)$ that differ by an $\varepsilon$-toggle at a vertex $v$, where $\varepsilon\in \{-1, 1\}$. That is, $g(v)=f(v)+\varepsilon$ and $g(w)=f(w)$ for all $w\neq v$. We denote by $h_{(f,g)}\in D(f,g)$ the unique labeling in $D(f, g)$ such that $h_{(f,g)}(v)=\varepsilon$ and $h_{(f,g)}(w)=0$ for all $w\neq v$. 

Given a walk $W = (c_0, \dots, c_r)$ in $\C_3(T)$, the two (directed) edges $(c_i, c_{i+1})$ and $(c_j, c_{j+1})$ are \emph{opposite} if they correspond to toggling the same vertex $v$ in opposite directions; this condition is equivalent to $h_{(c_{i}, c_{i+1})}=-h_{(c_{j}, c_{j+1})}$.

\begin{lem}\label{lem:opposite-edges}
No shortest path in $\C_3(T)$ contains opposite edges.
\end{lem}

\begin{proof}
Suppose, for a contradiction, that $P=(c_0,\ldots,c_r)$ is a shortest path in $\C_3(T)$ containing opposite edges. Among all such pairs, choose a pair $(c_i,c_{i+1})$ and $(c_j,c_{j+1})$ with $i<j$ such that the index difference $j-i$ is minimized. Let $v$ be the toggled vertex, with $c_{i+1}(v)=c_i(v)+\varepsilon$ and $c_{j+1}(v)=c_j(v)-\varepsilon$ for some $\varepsilon\in\{-1,1\}$.

If $j=i+1$, then $c_{i+2}=c_i$, and the path obtained by removing $c_{i+1}, c_{i+2}$ is shorter, a contradiction. This shows $j>i+1$. By the minimality of $j-i$, the vertex $v$ is not toggled between steps $i+1$ and $j$. Thus, $c_k(v)$ is constant for $i+1\le k\le j$.

For $c_i$ and $c_{i+1}$ to be adjacent proper colorings differing only at $v$, all neighbors of $v$ must share the same color $a$ in $c_i$. Moreover, $c_i(v)$ and $c_{i+1}(v)$ must be the two colors different from $a$. Since $c_{i+1}(v)-c_i(v)=\varepsilon$, this forces $c_i(v)=a+\varepsilon$ and $c_{i+1}(v)=a+2\varepsilon$. Similarly, for the toggle $c_j \to c_{j+1}$ (direction $-\varepsilon$) to be valid, all neighbors of $v$ must share the same color $a'$ in $c_j$. Since $c_{j+1}(v)-c_j(v)=-\varepsilon$, this forces $c_j(v)=a'-\varepsilon$ and $c_{j+1}(v)=a'-2\varepsilon$. Since $c_{i+1}(v)=c_j(v)$, we have $a+2\varepsilon=a'-\varepsilon$, which implies $a'=a+3\varepsilon\equiv a \pmod{3}$. Thus, the neighbors of $v$ have color $a$ in both $c_i$ and $c_j$.

Consider the subpath $P'=(c_{i+1}, \ldots, c_j)$. Throughout $P'$, $v$ is fixed at color $a+2\varepsilon$. If no neighbor of $v$ changes color during $P'$, then the sequence of toggles defining $P'$ is valid starting from $c_i$ (since the neighbors of $v$ remain colored $a$). Since $c_{j+1}(v) = a+\varepsilon = c_i(v)$, applying the same sequence to $c_i$ yields a parallel walk $(c_i, c'_{i+2}, \dots, c'_{j-1}, c_{j+1})$, where each $c'_k$ differs from $c_k$ only at $v$. This new walk skips the toggles at $v$. The resulting path from $c_0$ to $c_r$ is shorter than $P$ by $2$ edges, contradicting that $P$ is a shortest path.

Therefore, some neighbor $w$ must change color during $P'$. Since $v$ has color $a+2\varepsilon$, $w$ can only toggle between colors $a$ and $a+\varepsilon$. As $w$ starts and ends at color $a$, the path $P'$ must contain opposite toggles at $w$. These opposite toggles are strictly closer in index than the pair at $v$, contradicting the choice of the minimizing pair.
\end{proof}

While Lemma~\ref{lem:opposite-edges} restricts the form of a shortest path, the following lemma ensures that any compatible labeling can be realized by a walk with no unnecessary moves.

\begin{lem}\label{lem:path-norm}
For any $h\in D(f_1,f_2)$, there exists a walk from $f_1$ to $f_2$ in $\C_3(T)$ of length $\|h\|_1$ such that each vertex $w$ is toggled exactly $|h(w)|$ times, all in the direction $\operatorname{sign}(h(w))$. Consequently, the walk has no opposite edges.
\end{lem}

\begin{proof}
We proceed by induction on $n=|V|$. The base case $n=1$ is trivial.

For the inductive step, let $v$ be a leaf with neighbor $u$. Let $T'=T-\{v\}$ and let $h'$ be the restriction of $h$ to $T'$. By induction, there is a walk $W'$ in $\C_3(T')$ realizing $h'$. We construct a walk $W$ in $\C_3(T)$ by inserting the $|h(v)|$ required toggles of $v$ into $W'$.

First, consider the case where at least one of $h(v)$ or $h(u)$ is zero. If $h(u)=0$, $u$ is static in $W'$; we define $W$ by appending the toggle of $v$ (if $|h(v)|=1$) to the end of $W'$. If $h(v)=0$, we extend $W'$ to a walk $W$ in $\C_3(T)$ by keeping the color of $v$ fixed. 

Now assume that $h(v)$ and $h(u)$ are both non-zero. Since $|h(v)-h(u)|\le 1$, they must share the same sign $s \in \{-1, 1\}$. Let $\delta(c) = c(v)-c(u)$. For proper colorings, $\delta \in \{-s, s\}$. A \emph{paired toggle} replaces a single step $c(u) \to c(u)+s$ with a sequence of two steps that toggles both $u$ and $v$ by $s$. To maintain proper coloring, the order is determined by $\delta$: if $\delta=s$, we apply $(v, u)$; if $\delta=-s$, we apply $(u, v)$. Note that a paired toggle preserves the value of $\delta$.

We distinguish three cases based on the relative magnitudes of $|h(v)|$ and $|h(u)|$:

\textbf{Case 1:} $|h(v)|=|h(u)|$. Then $h(v)=h(u)$. The net change in $\delta$ is $h(v)-h(u) \equiv 0 \pmod{3}$, so $\delta_{\text{end}}=\delta_{\text{start}}$. We form $W$ by replacing every toggle of $u$ in $W'$ with a paired toggle sequence determined by the initial difference $\delta_{\text{start}}$.

\textbf{Case 2:} $|h(v)|=|h(u)|+1$. Then $h(v)=h(u)+s$. The net change in $\delta$ is $s \pmod{3}$. Since $\delta \in \{-s, s\}$, this forces $\delta_{\text{start}}=s$ and $\delta_{\text{end}}=-s$. We form $W$ by first applying a solo toggle of $v$. This changes $\delta$ from $s$ to $-s$. We then proceed as in Case 1, replacing every toggle of $u$ in $W'$ with a paired toggle of type $(u, v)$ (since $\delta$ is now $-s$).

\textbf{Case 3:} $|h(v)|=|h(u)|-1$. Then $h(v)=h(u)-s$. The net change in $\delta$ is $-s \pmod{3}$, which forces $\delta_{\text{start}}=-s$ and $\delta_{\text{end}}=s$. We replace the first $|h(v)|$ toggles of $u$ in $W'$ with paired toggles of type $(u, v)$ (since $\delta$ is initially $-s$). Since paired toggles preserve $\delta$, the value remains $-s$ just before the final toggle of $u$. We perform the final toggle of $u$ as a solo step, which changes $\delta$ from $-s$ to $s$.

In all cases, the constructed walk $W$ is valid and has length $\|h'\|_1 + |h(v)| = \|h\|_1$. Moreover, since every vertex moves monotonically (only increasing or only decreasing), the walk contains no opposite edges. \end{proof}

We are now ready to state the main result of this section, which identifies the distance in $\C_3(T)$ with the minimum norm of a compatible labeling.

\begin{thm}\label{thm:labelings} If $f,g$ are two $3$-colorings of $T$, viewed as vertices in $\C_3(T)$, then 
$$\min_{h\in D(f,g)}\|h\|_1=\dist_{\C_3(T)}(f,g).$$
\end{thm}

\begin{proof}
    Suppose that $c_0=f, c_1, \dots, c_r=g$ represents a path of shortest distance in $\C_3(T)$. Each edge $(c_j,c_{j+1})$ in the path corresponds to a single toggle, and recall that $h_{(c_j,c_{j+1})}$ is the labeling which vanishes everywhere except at the vertex being toggled. We can thus write
    $$\dist_{\C_3(T)}(f,g)=\sum_{j=0}^{r-1}\left\|h_{(c_j, c_{j+1})}\right\|_1$$
    where the right-hand side counts the total number of toggles. The path contains no opposite edges by Lemma~\ref{lem:opposite-edges}. Consequently, the toggles at each vertex accumulate monotonically without cancellation, implying
    \[
    \dist_{\C_3(T)}(f,g)=\left\|\sum_{j=0}^{r-1}h_{(c_j,c_{j+1})}\right\|_1.
    \]
    Observe that the sum $\sum_{j=0}^{r-1}h_{(c_j,c_{j+1})}$ is an element of $D(f,g)$. Therefore, 
    \[
    \dist_{\C_3(T)}(f,g)\geq \min_{h\in D(f,g)}\|h\|_1.
    \]
    On the other hand, for any $h\in D(f,g)$ there is a path from $f$ to $g$ of length $\|h\|_1$ by Lemma~\ref{lem:path-norm}. Hence,
    \[
    \dist_{\C_3(T)}(f,g)\leq \min_{h\in D(f,g)}\|h\|_1.
    \]
    Combining the two inequalities yields the desired result.
\end{proof}

In particular, the graph distance between two colorings is obtained by lifting their difference to $\mathbb{Z}$, and then minimizing the $L_1$-distance to $0$ over all such lifts.

Since balanced labelings are precisely the minimizers of the $L_1$-norm in their respective classes $D(f,g)$, we obtain the following useful characterization of the diameter.

\begin{cor}\label{cor:balanced}
    The diameter of $\C_3(T)$ is equal to the maximum value of $\|h\|_1$ as $h$ ranges over all balanced labelings of $T$.
\end{cor}

\section{Median values of balanced labelings}\label{sec:median}

In this section, we bound the median of a balanced labeling that attains the maximum possible $L_1$-norm. Our approach is to track how $\|h+x\|_1$ changes under integer shifts, combining this with the convexity of $\Phi(x)=\|h+x\|_1$ and the balanced condition \eqref{eq:convex-balanced}.

Recall that a balanced labeling of a tree $T$ is a function $h\colon V\to \mathbb{Z}$ satisfying two conditions:
\begin{itemize}
    \item $|h(u)-h(v)|\leq 1$ for each edge $uv\in E(T)$.
    \item $\|h\|_1 \le \|h+3m\|_1$ for all $m\in \Z$.
\end{itemize}
The second bullet point motivates us to understand the behavior of $h+3m$ as $m$ varies. We analyze these shifts using the function $\Phi(x) \colonequals \|h+x\|_1 = \sum_{v\in V} |h(v)+x|$. Since $\Phi$ is a sum of convex functions $x \mapsto |c+x|$, $\Phi$ is convex. Since $h$ is integer-valued, the multiset of labels $\{h(v)\}$ has the median $M\in \frac{1}{2} \mathbb{Z}$. The following notation helps with convexity arguments.

\textbf{Notation.} Given a labeling $h\colon V\to \Z$, consider it as a multiset $\{h(v)\}$ of labels. Define the following quantities:
\begin{itemize}
    \item $a_i$ is the number of vertices with label equal to $i$.
    \item $A_{\geq t}$ (resp. $A_{\leq t}$) is the number of vertices with label at least $t$ (resp. at most $t$).
    \item $\Delta_k \colonequals \Phi(k+1)-\Phi(k)$ represents the forward difference (slope) at $x=k$.
\end{itemize}
We observe that
\[
\Delta_k = \#\{v : h(v)+k \ge 0\} - \#\{v : h(v)+k < 0\} = A_{\ge -k} - (n - A_{\ge -k}) = 2A_{\ge -k} - n.
\]
The identities $\Delta_k = 2A_{\geq -k}-n$, and equivalently, $A_{\geq -k} = \frac{\Delta_k+n}{2}$ are useful in our proofs. The following lemma establishes a connection between the median of the labels of $h$ and minimizers of the function $\Phi(x)$. 

\begin{lem}[Median and minimizer]
\label{lem:minimizer}
Let $M$ be the median of the labels of $h$. An integer $k \in \Z$ minimizes $\Phi(k)$ if and only if $-k\in [M-\frac{1}{2}, M+\frac{1}{2}]$. In particular, if $|M| < 1$, then $0$ is a global minimizer of $\Phi$.
\end{lem}

\begin{proof}
By convexity, $k$ minimizes $\Phi$ if and only if $\Delta_{k-1} \le 0 \le \Delta_k$. Using the formula for $\Delta_k$ above, this is equivalent to
\[
2A_{\ge -(k-1)} - n \le 0 \le 2A_{\ge -k} - n \quad \iff \quad A_{\ge -k+1} \le n/2 \le A_{\ge -k}.
\]
This is equivalent to the condition that $-k\in \left[M-\frac{1}{2}, M+\frac{1}{2}\right]$. \end{proof}

We also record a simple property of convex functions restricted to a lattice.

\begin{lem}
\label{lem:lattice}
Let $f\colon \Z \to \R$ be convex with a global minimum at $0$. Let $L \subset \Z$ be a lattice (i.e., an arithmetic progression) such that $0 \notin L$, containing points strictly less than and strictly greater than $0$. Then $\min_{x \in L} f(x)$ is attained at one of the two lattice points $u < 0 < v$ of $L$ adjacent to $0$.
\end{lem}

\begin{proof}
Note that $f$ is non-increasing on $(-\infty, 0]$ and non-decreasing on $[0, \infty)$. For any $z \in L$ with $z < u$, $f(z) \ge f(u)$; for $z > v$, $f(z) \ge f(v)$. Thus, the minimum is at $u$ or $v$.
\end{proof}

The next lemma is the main result of this section. This lemma helps narrow down the search for balanced labelings of maximum norm.

\begin{lem}\label{lem:median-control}
Suppose $h$ is a balanced labeling of $T$ with maximum $L_1$-norm (among all balanced labelings of $T$). Let $M$ be the median of the labels of $h$. Then $M\in \{1, \frac{3}{2}, 2\}$.
\end{lem}

\begin{proof} By symmetry (replacing $h$ with $-h$ if necessary), we assume $M \ge 0$. Recall that $a_i$ denotes the number of vertices with label $i$.

\medskip\noindent
\textbf{Upper Bound ($M \le 2$).}
Since $h$ is balanced, we have $\Phi(0) \le \Phi(-3)$. We expand the difference 
\[
\Phi(0) - \Phi(-3) = \sum_{i} a_i (|i| - |i-3|)\leq 0
\]
using:
\begin{itemize}[noitemsep]
    \item $i \le 0 \implies |i| - |i-3| = -3$; \quad $i \ge 3 \implies |i| - |i-3| = 3$.
    \item $i = 1 \implies |i| - |i-3| = -1$; \quad $i = 2 \implies |i| - |i-3| = 1$;
\end{itemize}
Thus, $-3A_{\le 0} - a_1 + a_2 + 3A_{\ge 3} \le 0$. Substituting $A_{\le 0} = n - a_1 - a_2 - A_{\ge 3}$ yields:
\[
-3n + 2a_1 + 4a_2 + 6A_{\ge 3} \le 0 \implies 6A_{\ge 3} \le 3n - 2a_1 - 4a_2 \le 3n.
\]
This implies $A_{\ge 3} \le n/2$. If $M > 2$, then $M \ge 3$ and so $A_{\ge 3} > n/2$, a contradiction.

\medskip\noindent
\textbf{Lower Bound ($M \ge 1$).}
Suppose $|M| < 1$. By Lemma \ref{lem:minimizer}, we know that $x=0$ is a global minimizer of $\Phi$, so slopes satisfy $\Delta_{-1} \le 0 \le \Delta_0$. By maximality, if $\|h+k\|_1 > \|h\|_1$, then $h+k$ is unbalanced.

\smallskip
\textbf{Case 1:} Strict minimum ($\Delta_{-1} < 0 < \Delta_0$).
Since $\|h\pm 1\|_1 > \|h\|_1$, both $h+1$ and $h-1$ are unbalanced.
For $h+1$ (on lattice $3\Z+1$ straddling $-2$ and $1$):
\[
\text{$h+1$ unbalanced} \implies \min(\Phi(-2), \Phi(1)) < \Phi(1) \implies \Phi(-2) < \Phi(1).
\]
For $h-1$ (on lattice $3\Z-1$ straddling $-1$ and $2$):
\[
\text{$h-1$ unbalanced} \implies \min(\Phi(-1), \Phi(2)) < \Phi(-1) \implies \Phi(2) < \Phi(-1).
\]
Summing these gives $\Phi(1) - \Phi(2) > \Phi(-2) - \Phi(-1)$, or $\Delta_1 < \Delta_{-2}$, contradicting convexity.

\smallskip
\textbf{Case 2:} Flat minimum (at least one slope is 0).
By symmetry ($h \mapsto -h$), we may assume $\Delta_{-1} = 0 \le \Delta_0$. This implies $A_{\ge 1} = (\Delta_{-1}+n)/2 = n/2$.
\begin{itemize}
    \item If $\Delta_0 = 0$, then $A_{\ge 0} = n/2$, so $a_0 = 0$. Since $A_{\ge 1} = n/2 > 0$ and $A_{\le -1} = n-A_{\geq 1} > 0$, the label set is \emph{not} contiguous, a contradiction.
    
    \item If $\Delta_0 > 0$, then $a_0 > 0$ and $a_1 > 0$. We compute the relevant slopes:
    \begin{align*}
        \Delta_0 &= 2A_{\ge 0} - n = 2(A_{\ge 1} + a_0) - n = 2a_0 > 0, \\ 
        \Delta_{-2} &= 2A_{\ge 2} - n = 2(A_{\ge 1} - a_1) - n = -2a_1 < 0.
    \end{align*}
    Thus, $\|h+1\|_1 > \|h\|_1$ and $\|h-2\|_1 > \|h\|_1$. By maximality, both are unbalanced on the lattice $3\Z+1$. Applying Lemma~\ref{lem:lattice}, we obtain:
    \[
       \Phi(-2) < \Phi(1) \quad (\text{from } h+1) 
       \quad \text{and} \quad
       \Phi(1) < \Phi(-2) \quad (\text{from } h-2),
    \]
    a contradiction.
\end{itemize}
Thus, the lower bound $|M| \ge 1$ holds. \end{proof}

\section{Maximum 3-coloring diameter: path graph}\label{sec:maximum-3-coloring-diameter}

Let $I_n \colonequals [2-\lceil n/2\rceil, \lfloor n/2\rfloor+1]\cap\mathbb{Z}$ denote the set of $n$ consecutive integers. We begin by showing that the path graph is the unique tree admitting a labeling by this set.

\begin{lem}\label{lem:path-unique}
If a tree $T$ on $n$ vertices admits a labeling $h\colon V(T)\to \mathbb{Z}$ whose image is exactly the set $I_n$, then $T \cong P_n$.
\end{lem}

\begin{proof}
Let $P_n$ be the path graph on vertex set $I_n$ with edges between consecutive integers. Since the image of $h$ is $I_n$ and $|h(u)-h(v)|\le 1$ for each edge $uv$ of $T$, the map $h$ is a graph homomorphism from $T$ to $P_n$ that is bijective on vertices. Thus $T$ is a spanning subgraph of $P_n$. Since $P_n$ is a tree, its only connected spanning subgraph is $P_n$ itself. Thus $T \cong P_n$.
\end{proof}

The path graph $P_n$ admits such a labeling by assigning the integers of $I_n$ sequentially along the path. For $n=7$, we have $I_7 = \{-2, -1, 0, 1, 2, 3, 4\}$, as illustrated below:

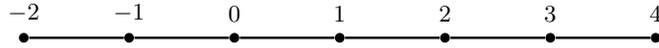
\tikzset{
  smvtx/.style={circle,fill=black,inner sep=1.3pt},
  smedge/.style={line width=0.9pt},
  lab/.style={font=\scriptsize, inner sep=1pt}
}

\begin{center}
\begin{tikzpicture}[scale=1.4]
  \node[smvtx] (v0) at (0,0) {};
  \node[smvtx] (v1) at (1,0) {};
  \node[smvtx] (v2) at (2,0) {};
  \node[smvtx] (v3) at (3,0) {};
  \node[smvtx] (v4) at (4,0) {};
  \node[smvtx] (v5) at (5,0) {};
  \node[smvtx] (v6) at (6,0) {};

  \draw[smedge] (v0)--(v1);
  \draw[smedge] (v1)--(v2);
  \draw[smedge] (v2)--(v3);
  \draw[smedge] (v3)--(v4);
  \draw[smedge] (v4)--(v5);
  \draw[smedge] (v5)--(v6);

  \node[lab, above=3pt of v0] {$-2$};
  \node[lab, above=3pt of v1] {$-1$};
  \node[lab, above=3pt of v2] {$0$};
  \node[lab, above=3pt of v3] {$1$};
  \node[lab, above=3pt of v4] {$2$};
  \node[lab, above=3pt of v5] {$3$};
  \node[lab, above=3pt of v6] {$4$};
\end{tikzpicture}
\end{center}

\begin{thm}\label{thm:path} For every tree $T$ on $n\geq 7$ vertices, the diameter of $\C_3(T)$ is at most $$\binom{\lceil\frac{n}{2}\rceil -1}{2}+\binom{\lfloor\frac{n}{2}\rfloor+2}{2}$$
and equality is achieved if and only if $T$ is isomorphic to the path graph $P_n$.\end{thm}

\begin{proof}

Let $h\colon V(T)\to \mathbb{Z}$ be a balanced labeling of maximum $L_1$-norm, and let $M$ be the median of its labels. By Lemma~\ref{lem:median-control}, we have $M\in \{1, \frac{3}{2}, 2\}$. As before, let $a_i$ be the multiplicity of each label $i$. Define 
\[
S_{\max} \colonequals \sum_{x\in I_n} |x| = \binom{\lceil n/2\rceil -1}{2} + \binom{\lfloor n/2\rfloor+2}{2}.
\]
We show that $\|h\|_1 \le S_{\max}$, with equality if and only if the image of $h$ is exactly $I_n$. We proceed according to two cases depending on the value of $M$.

\textbf{Case 1:} $M=1$ or $M=\frac{3}{2}$.

In either case, at least $\lceil n/2\rceil$ labels are $\le 1$ and at least $\lceil n/2\rceil$ are $\ge 1$. Let $p = \sum_{i\ge 2} a_i$ and $q = \sum_{i\le 0} a_i$. Then $p \le \lfloor n/2 \rfloor$ and $q \le \lceil n/2 \rceil - 1$. We decompose the difference as follows:

\begin{align*}
S_{\max}-\|h\|_1&=\left(\sum_{2-\lceil \frac{n}{2}\rceil\le i\le \lfloor\frac{n}{2}\rfloor+1}|i|\right)-\sum_{i}|i|a_i  \\
&=\left(\sum_{i=2}^{p+1} i-\sum_{i\geq 2}ia_i \right)+\left(\sum_{i=1-q}^{0} |i|-\sum_{i\le 0}|i|a_i \right)+\left(\sum_{\substack{2-\lceil \frac{n}{2}\rceil\le i\le -q \\ 
\text{or} \\
p+2\le i\le \lfloor\frac{n}{2}\rfloor+1}} (|i|-1) \right).
\end{align*}
 We claim that the first two summands are always nonnegative, with equality if and only if $a_i=1$ in the appropriate intervals. To see this, imagine ``flattening'' the labels that have values $\geq 2$ by replacing each label of multiplicity more than one with the next smallest value that is not achieved by a label. The second summand can be handled in the same way by flattening the values towards the negative side. The final summand is nonnegative for $n\geq 7$ because $I_n$ is wide enough: the only term that could possibly be negative is $(|0|-1)$ which appears when $q=0$. However, $2-\lceil\frac{n}{2}\rceil\le -2$, so all three terms $(|-2|-1)+(|-1|-1)+(|0|-1)$ combined give a nonnegative contribution. 
 
\textbf{Case 2:} $M=2$.

Since $2$ is the median, it follows that $\sum_{i\geq 2}a_i>\sum_{i\le 1}a_i$, and therefore we obtain
\begin{equation}\label{ineq:path-1}
a_2 + \sum_{i\geq 3}a_i > a_1 + \sum_{i\le 0} a_i \, .
\end{equation}
 Moreover, since $h$ is a balanced labeling, we have $\|h\|_1\le \|h-3\|_1$ which yields
\begin{equation}\label{ineq:path-2}
a_2+3\sum_{i\geq 3}a_i\le a_1+3\sum_{i\le 0}a_i \, .
\end{equation}
The inequalities~\eqref{ineq:path-1} and \eqref{ineq:path-2} together imply that $\sum_{i\le 0} a_i > \sum _{i\geq 3} a_i$ and $a_2>a_1$. Let $p=\sum_{i\geq 3} a_i$, and $q=\sum_{i\le 1} a_i$. Notice that we have the bounds
\[
p+2<\sum_{i\le 0}a_i+2\le \sum_{i\le 1}a_i+1\le \left\lfloor\frac{n}{2}\right\rfloor+1 \quad \text{ and } \quad  2-q > 2-\left\lceil\frac{n}{2}\right\rceil.
\]
We express the difference:
$$
S_{\max}-\|h\|_1=\left(\sum_{2-\lceil \frac{n}{2}\rceil\le i\le \lfloor\frac{n}{2}\rfloor+1}|i|\right)-\sum_{i}|i|a_i
$$
$$=\left(\sum_{i=3}^{p+2} i-\sum_{i\geq 3}ia_i \right)+\left(\sum_{i=2-q}^{1} |i|-\sum_{i\le 1}|i|a_i \right)+\left(\sum_{\substack{2-\lceil \frac{n}{2}\rceil\le i\le 1-q \\ 
\text{or} \\
p+3\le i\le \lfloor\frac{n}{2}\rfloor+1}} \left(|i|-2\right) \right).
 $$
As in Case~1, the first two summands are nonnegative, with equality only when the labels in the corresponding intervals appear with multiplicity one and occupy consecutive integer values (this is the same flattening argument applied to the positive and negative sides separately). The third summand is nonempty because $a_2>1$. Moreover, the terms are all nonnegative except possibly the contribution from $i=-1$, namely $|-1|-2 = -1$; this term can occur only if $a_0=1$ and $a_i=0$ for all $i\le -1$. Since $\sum_{i\geq 3}a_i< \sum_{i\le 0} a_i$, we must have $a_3=0$, and so the term $(|3|-2)$ also appears. However, $\lfloor \frac{n}{2}\rfloor+1 \geq 4$ for $n\geq 7$, so the presence of the term $(|4|-2)$ would make the inequality strict. 

Therefore, $\|h\|_1 \le S_{\max}$, with equality exactly when the label multiset equals $I_n$. By Lemma~\ref{lem:path-unique}, this implies $T \cong P_n$.
\end{proof}

Theorem~\ref{thm:path} is precisely the first part of our main Theorem~\ref{thm:main}.

\section{Minimal 3-coloring diameter: star and nearly symmetric double stars}\label{sec:minimum-3-coloring-diameter}

In this section, we prove the second part of our main theorem: the trees with the minimum 3-coloring diameter are precisely the star graph $K_{1, n-1}$ and the double stars $S(a, b)$ with $|a-b| \le 4$. We first establish a universal lower bound for all trees (Lemma~\ref{lem:universal-lower-bound}), and then show that these trees, and only these trees, attain it.

Throughout this section, we write $u\sim v$ to denote that vertices $u$ and $v$ are adjacent. For a vertex $v$, we write $N(v)$ for its (open) neighborhood, 
$N_i(v)$ for the set of vertices at distance exactly $i$ from $v$, 
and $N_{\ge i}(v)$ for those at distance at least $i$.

\begin{lem}\label{lem:universal-lower-bound}
For each tree $T$ on $n$ vertices, we have $\diam(\C_3(T))\geq \lfloor 3n/2\rfloor$.
\end{lem}

\begin{proof} By Corollary~\ref{cor:balanced}, it suffices to find a balanced labeling $h\colon V\to \Z$ with $\|h\|_1 = \lfloor 3n/2\rfloor$. Partition the vertex set $V$ arbitrarily into two subsets $U_1$ and $U_2$ with $|U_1|=\lceil n/2 \rceil$ and $|U_2|=\lfloor n/2 \rfloor$. Consider the labeling $h$ given by $h(u)=1$ for all $u\in U_1$ and $h(u)=2$ for all $u\in U_2$. The condition $|h(u)-h(v)|\leq 1$ for all $uv\in E(T)$ is trivially satisfied. The norm is $\|h\|_1=\lceil n/2\rceil + 2 \lfloor n/2\rfloor = \lfloor 3n/2\rfloor$. Finally, to see that $h$ is balanced, we compute
\begin{align*}
\|h+3\|_1 &= \sum_{u\in U_1} |1+3| + \sum_{u\in U_2} |2+3| = 4\lceil n/2\rceil + 5\lfloor n/2\rfloor = \lfloor 9n/2 \rfloor > \lfloor 3n/2\rfloor = \|h\|_1,  \\ 
\|h-3\|_1 &= \sum_{u\in U_1} |1-3| + \sum_{u\in U_2} |2-3|= 2\lceil n/2\rceil + \lfloor n/2 \rfloor = \lceil 3n/2 \rceil \geq \lfloor 3n/2\rfloor = \|h\|_1.
\end{align*}
By equation~\eqref{eq:convex-balanced}, the labeling $h$ is balanced. Thus, $\diam(\C_3(T))$ is at least $\lfloor 3n/2\rfloor$.
\end{proof}

This lemma shows that $\lfloor 3n/2\rfloor$ is a universal lower bound for the $3$-coloring diameter of a tree. In the following lemmas, we prove that any tree that is not a star or a nearly symmetric double star has a strictly larger $3$-coloring diameter.

\medskip 

\begin{lem}\label{lem:small-2-neighborhood}
Let $T$ be a tree on $n$ vertices. If there exists a leaf vertex $v$ with $|N_2(v)| \le \lceil n/2\rceil-4$, then $\diam(\C_3(T)) \ge \lfloor 3n/2\rfloor + 1$.
\end{lem}

\begin{proof}
We construct a balanced labeling $h$ with $\|h\|_1 = \lfloor 3n/2 \rfloor + 1$. Let $u$ be the unique neighbor of $v$. Let $s_2 = |N_2(v)|$. Define $h\colon V \to \Z$ as follows:
\begin{itemize}
    \item $h(v) = -1$ and $h(u) = 0$.
    \item $h$ assigns the value $1$ to a set of $n_1 = \lceil n/2 \rceil - 4$ vertices, chosen to include all $s_2$ vertices in $N_2(v)$.
    \item $h$ assigns the value $2$ to all $n_2 = \lfloor n/2 \rfloor + 2$ remaining vertices.
\end{itemize}
This construction is well-defined since $s_2 \le \lceil n/2 \rceil - 4$ by hypothesis. The labeling is valid: the edge $v \sim u$ is valid ($|-1-0|=1$), edges $u \sim w$ for $w \in N_2(v)$ are valid ($|0-1|=1$), and all other edges $x \sim y$ must have $h(x), h(y) \in \{1, 2\}$, so $|h(x)-h(y)| \le 1$.

The norm of $h$ is
\begin{align*}
\|h\|_1 &= 1\cdot|-1| + 1\cdot|0| + n_1 \cdot |1| + n_2 \cdot |2| \\
&= 1 + 0 + (\lceil n/2 \rceil - 4) \cdot 1 + (\lfloor n/2 \rfloor + 2) \cdot 2 \\
&= 1 + \lceil n/2 \rceil + 2\lfloor n/2 \rfloor = 1 + n + \lfloor n/2 \rfloor = \lfloor 3n/2 \rfloor + 1.
\end{align*}
To check that $h$ is balanced, we compute $\|h-3\|_1$:
\begin{align*}
\|h-3\|_1 &= 1\cdot|-4| + 1\cdot|-3| + n_1 \cdot |-2| + n_2 \cdot |-1| \\
&= 4 + 3 + (\lceil n/2 \rceil - 4) \cdot 2 + (\lfloor n/2 \rfloor + 2) \cdot 1 \\
&= 1 + 2\lceil n/2 \rceil + \lfloor n/2 \rfloor = 1+\lceil n/2 \rceil + n = \lceil 3n/2 \rceil+1.
\end{align*}
Since $\|h-3\|_1 = 1 + \lceil 3n/2 \rceil \ge 1 + \lfloor 3n/2 \rfloor = \|h\|_1$, and $\|h+3\|_1$ is strictly larger (as $|h(x)+3| > |h(x)|$ for all $x\in V$), we conclude that $h$ is balanced by equation~\eqref{eq:convex-balanced}. Therefore, $\diam(\C_3(T)) \ge \|h\|_1 = \lfloor 3n/2 \rfloor + 1$ by Corollary~\ref{cor:balanced}.
\end{proof}

\medskip 

\begin{lem}\label{lem:large-2-neighborhood}
Let $T$ be a tree on $n$ vertices. If there exists a leaf $v$ and an integer $k \ge 1$ such that $|N_2(v)| = \lceil n/2 \rceil - 4 + k$ and $|N_{\ge 4}(v)| \ge k$, then $\diam(\mathcal C_3(T)) \ge \lfloor 3n/2\rfloor + 1$.
\end{lem}

\begin{proof}
We construct a balanced labeling $h$ with $\|h\|_1 = \lfloor 3n/2\rfloor + 1$. Let $u$ be the unique neighbor of $v$. Let $s_2 = |N_2(v)|$. By hypothesis, $s_2 = \lceil n/2 \rceil - 4 + k$. Define $h\colon V \to \Z$ as follows:
\begin{itemize}
    \item $h(v) = -1$ and $h(u) = 0$.
    \item $h$ assigns the value $1$ to all $s_2$ vertices in $N_2(v)$.
    \item $h$ assigns the value $3$ to a set of $k$ vertices in $N_{\ge 4}(v)$ (which exist by hypothesis).
    \item $h$ assigns the value $2$ to all $n_2=\lfloor n/2 \rfloor + 2 - 2k$ remaining vertices.
\end{itemize}
This construction is well-defined, as the count $n_2$ for the remaining vertices is
\[ n_2 = n - 2 - s_2 - k = n - 2 - (\lceil n/2 \rceil - 4 + k) - k = \lfloor n/2 \rfloor + 2 - 2k. \]
The labeling is valid: the edge $v \sim u$ is valid ($|-1-0|=1$), edges $u \sim w$ for $w \in N_2(v)$ are valid ($|0-1|=1$), and all other edges $x \sim y$ must have $h(x), h(y) \in \{1, 2, 3\}$ and satisfy $|h(x)-h(y)| \le 1$, since vertices in $N_3(v)$ are assigned label 2.

The norm of $h$ is
\begin{align*}
\|h\|_1 &= 1\cdot|-1| + 1\cdot|0| + s_2 \cdot |1| + k \cdot |3| + n_2 \cdot |2| \\
&= 1 + 0 + (\lceil n/2 \rceil - 4 + k) \cdot 1 + 3k + (\lfloor n/2 \rfloor + 2 - 2k) \cdot 2 \\
&= 1 + \lceil n/2 \rceil - 4 + k + 3k + 2\lfloor n/2 \rfloor + 4 - 4k \\
&= 1 + \lceil n/2 \rceil + 2\lfloor n/2 \rfloor = 1 + n + \lfloor n/2 \rfloor = \lfloor 3n/2 \rfloor + 1.
\end{align*}
To check that $h$ is balanced, we compute $\|h-3\|_1$:
\begin{align*}
\|h-3\|_1 &= 1\cdot|-4| + 1\cdot|-3| + s_2 \cdot |-2| + k \cdot |0| + n_2 \cdot |-1| \\
&= 4 + 3 + (\lceil n/2 \rceil - 4 + k) \cdot 2 + 0 + (\lfloor n/2 \rfloor + 2 - 2k) \cdot 1 \\
&= 7 + 2\lceil n/2 \rceil - 8 + 2k + \lfloor n/2 \rfloor + 2 - 2k \\
&= 1 + 2\lceil n/2 \rceil + \lfloor n/2 \rfloor = 1 + \lceil n/2 \rceil + n = \lceil 3n/2 \rceil + 1.
\end{align*}
Since $\|h-3\|_1 = 1 + \lceil 3n/2 \rceil \ge 1 + \lfloor 3n/2 \rfloor = \|h\|_1$, and $\|h+3\|_1$ is strictly larger (as $|h(x)+3| > |h(x)|$ for all $x\in V$), we conclude that $h$ is balanced by equation~\eqref{eq:convex-balanced}. Therefore, $\diam(\C_3(T)) \ge \|h\|_1 = \lfloor 3n/2 \rfloor + 1$ by Corollary~\ref{cor:balanced}.
\end{proof}

\begin{figure}[ht]
\centering
\begin{tikzpicture}[
    vertex/.style={circle, draw, fill=black, inner sep=1.5pt},
    edge/.style={thick},
    label distance=2pt
]

\node[vertex, label=below:$v$] (v) at (0, 0) {};
\node[vertex, label=below:$u$] (u) at (2, 0) {};
\draw[edge] (v) -- (u);

\node[vertex] (n2_1) at (4, 1.5) {};
\node[vertex] (n2_2) at (4, 0) {};
\node[vertex] (n2_3) at (4, -1.5) {};
\draw[edge] (u) -- (n2_1);
\draw[edge] (u) -- (n2_2);
\draw[edge] (u) -- (n2_3);
\node at (4, 2.7) {$N_2(v)$};
\draw[dashed, color=gray] (4, 0.3) ellipse (0.7cm and 2.1cm);

\node[vertex] (n3_1) at (6, 1.5) {};
\node[vertex] (n3_2) at (6, 0) {};
\draw[edge] (n2_1) -- (n3_1);
\draw[edge] (n2_2) -- (n3_2);
\node at (6, 2.7) {$N_3(v)$};
\draw[dashed, color=gray] (6, 0.75) ellipse (0.7cm and 1.6cm);

\node[vertex] (n4_1) at (8, 1.5) {};
\node[vertex] (n4_2) at (8, 0.3) {};
\node[vertex] (w) at (8, -0.7) {}; 
\draw[edge] (n3_1) -- (n4_1);
\draw[edge] (n3_2) -- (n4_2);
\draw[edge] (n3_2) -- (w); 
\node at (8, 2.7) {$N_{\ge 4}(v)$};
\draw[dashed, color=gray] (8, 0.25) ellipse (0.7cm and 2.1cm);

\node[above=4pt of v, font=\bfseries] {$-1$};
\node[above=4pt of u, font=\bfseries] {$0$};
\node[above=4pt of n2_1, font=\bfseries] {$1$};
\node[above=4pt of n2_2, font=\bfseries] {$1$};
\node[above=4pt of n2_3, font=\bfseries] {$1$};
\node[above=4pt of n3_1, font=\bfseries] {$2$};
\node[above=4pt of n3_2, font=\bfseries] {$2$};
\node[above=4pt of n4_1, font=\bfseries] {$3$};
\node[above=4pt of n4_2, font=\bfseries] {$3$};
\node[below=4pt of w, font=\bfseries] {$2$};

\end{tikzpicture}
\caption{A balanced labeling $h\colon V \to \Z$ with $\|h\|_1 = \lfloor 3n/2 \rfloor + 1$. The numbers shown are the values $h(x)$. This tree on $n=10$ vertices has $|N_2(v)|=3$ and $|N_{\ge 4}(v)|=3$. This satisfies the hypothesis of Lemma~\ref{lem:large-2-neighborhood} for $k=2$ (since $|N_2(v)| = 3 = \lceil 10/2 \rceil - 4 + 2$ and $|N_{\ge 4}(v)| = 3 \ge 2$). The labeling $h$ assigns $3$ to $k=2$ vertices in $N_{\ge 4}(v)$ and $2$ to all other vertices in $N_{\ge 3}(v)$.}
\label{fig:labeling}
\end{figure}
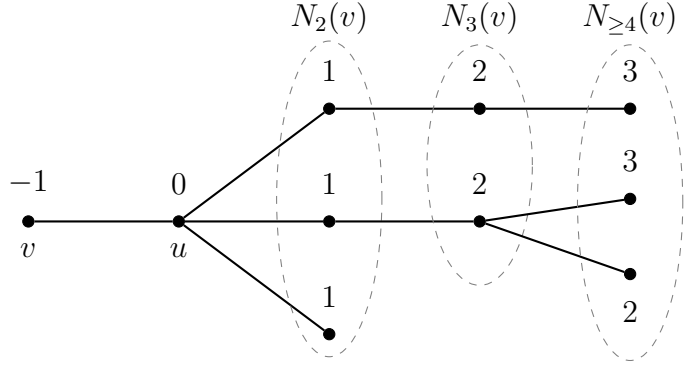

\begin{lem}\label{lem:unbalanced-double-star}
    If $T=S(a, b)$ is a double star with $|a-b|> 4$, then $\diam(\C_3(T))>\lfloor \frac{3n}{2}\rfloor$.
\end{lem}

\begin{proof}
Assume, to the contrary, that $\diam(\C_3(T)) = \lfloor 3n/2\rfloor$. By Lemma~\ref{lem:small-2-neighborhood}, this implies that \emph{every} leaf $v$ must satisfy $|N_2(v)| > \lceil n/2 \rceil - 4$, that is, $|N_2(v)| \geq \lceil n/2 \rceil - 3$.

Let $u_1$ be the center with $a$ leaves, and $u_2$ be the center with $b$ leaves. By hypothesis, $|a-b| \ge 5$, so we may assume $b \ge a+5$.
Let $v$ be a leaf attached to $u_1$. The set $N_2(v)$ consists of the other $a-1$ leaves on $u_1$ and the center $u_2$, so $|N_2(v)| = a$.
By our assumption, $a \geq \lceil n/2 \rceil - 3$. 

Combining $n = a+b+2$ and $b \ge a+5$ yields
\[
n \ge a + (a+5) + 2 = 2a+7.
\]
Now, using $a \ge \lceil n/2 \rceil - 3$, we have
\[
n \ge 2a+7 \ge 2(\lceil n/2 \rceil - 3) + 7 = 2\lceil n/2 \rceil + 1 \geq n+1,
\]
a contradiction.
\end{proof}

\begin{lem}\label{lem:reduction-diameter-3}
   Suppose $T$ is a tree on $n \ge 7$ vertices such that for every leaf $v$, $|N_2(v)| > \lceil n/2 \rceil - 4$ and $|N_{\ge 4}(v)| < |N_2(v)| + 4 - \lceil n/2 \rceil$. Then $\diam(T) \le 3$.
\end{lem}

\begin{proof}
    Suppose, towards a contradiction, that $\diam(T)\ge 4$. Then there exist leaves $v, w$ with $\dist_T(v,w)=\delta \ge 4$. Without loss of generality, assume $|N_2(w)| \ge |N_2(v)|$. Let $P$ be a shortest path between $v$ and $w$.

    \textbf{Case 1:} $\dist_T(v,w)=4$. Let the path $P$ be $v \sim u \sim z \sim x \sim w$. Then $N_2(w) = N(x) \setminus \{w\}$, so $|N_2(w)| = \deg(x)-1$. The set $N_{\ge 4}(v)$ contains all vertices in $N(x) \setminus \{z\}$, since these are at distance $4$ from $v$. Thus,
\begin{align}\label{ineq1:diam-3-reduction}
 |N_{\ge 4}(v)| \ge |N(x) \setminus \{z\}| = \deg(x)-1 = |N_2(w)| \ge |N_2(v)|.
\end{align}
However, the lemma hypothesis states that:
\begin{equation}\label{ineq2:diam-3-reduction}
    |N_{\ge 4}(v)| < |N_2(v)| + 4 - \lceil n/2 \rceil
\end{equation}
Combining the two inequalities \eqref{ineq1:diam-3-reduction} and \eqref{ineq2:diam-3-reduction} yields $4-\lceil n/2\rceil > 0$, 
a contradiction for $n\geq 7$.

    \textbf{Case 2:} $\dist_T(v, w)\geq 5$. Let $x$ be the neighbor of $w$ on $P$, and let $y$ be the neighbor of $x$ preceding it on $P$. As before, $|N_2(w)| = \deg(x)-1$. The set $N_{\ge 4}(v)$ contains the vertex $w$ (at distance $\delta \ge 5$) and the vertex $x$ (at distance $\delta-1 \ge 4$). It also contains the vertices in $N(x) \setminus \{w, y\}$, all of which are at distance $\delta$ from $v$. Thus,
\begin{align}\label{ineq3:diam-3-reduction}
    |N_{\ge 4}(v)| \geq 1 + 1 + (\deg(x)-2) = |N_2(w)|+1 \geq |N_2(v)|+1.
\end{align}
Combining inequality~\eqref{ineq3:diam-3-reduction} with the lemma's hypothesis \eqref{ineq2:diam-3-reduction} yields $3-\lceil n/2\rceil>0$, a contradiction for $n\geq 7$.
\end{proof}

\begin{cor}\label{cor:minimum-3-coloring-diameter}
If $n \ge 7$ and a tree $T$ has minimum 3-coloring diameter $\lfloor 3n/2\rfloor$, then $T$ is either a star or a double star $S(a,b)$ with $|a-b|\le 4$.
\end{cor}
\begin{proof}
By Lemma~\ref{lem:universal-lower-bound}, every tree on $n$ vertices satisfies $\diam(\C_3(T)) \ge \lfloor 3n/2 \rfloor$. We show that trees not in the stated class have strictly larger diameter. Assume $T$ is neither a star nor a double star $S(a,b)$ with $|a-b|\le 4$. We consider two cases.

\textbf{Case 1:} $T = S(a,b)$ with $|a-b| > 4$. By Lemma~\ref{lem:unbalanced-double-star}, $\diam(\C_3(T)) > \lfloor 3n/2 \rfloor$.

\textbf{Case 2:} $T$ is not a star or double star. Then $\diam(T) \ge 4$. By the contrapositive of Lemma~\ref{lem:reduction-diameter-3}, there exists a leaf $v$ such that at least one of the following holds:
\begin{enumerate}
    \item[(A)] $|N_2(v)| \le \lceil n/2 \rceil - 4$, or
    \item[(B)] $|N_{\ge 4}(v)| \ge |N_2(v)| + 4 - \lceil n/2 \rceil$.
\end{enumerate}

If (A) holds, then by Lemma~\ref{lem:small-2-neighborhood}, $\diam(\C_3(T)) \ge \lfloor 3n/2 \rfloor + 1$.

If (A) does not hold but (B) holds, let $k = |N_2(v)| + 4 - \lceil n/2 \rceil$. Since (A) fails, we have $|N_2(v)| > \lceil n/2 \rceil - 4$, which implies $k \ge 1$. By rearranging the definition of $k$, we obtain $|N_2(v)| = \lceil n/2 \rceil - 4 + k$. Combined with (B), which gives $|N_{\ge 4}(v)| \ge k$, the conditions of Lemma~\ref{lem:large-2-neighborhood} are satisfied. Thus, $\diam(\C_3(T)) \ge \lfloor 3n/2 \rfloor + 1$.

In both cases, $\diam(\C_3(T)) > \lfloor 3n/2 \rfloor$, completing the proof.
\end{proof}

\begin{prop}\label{prop:star}
The 3-coloring diameter of $K_{1, n-1}$ is $\lfloor 3n/2 \rfloor$.
\end{prop}
\begin{proof}
The lower bound follows from Lemma~\ref{lem:universal-lower-bound}. For the upper bound, it suffices to show that any two 3-colorings are at distance at most $\lfloor 3n/2 \rfloor$.

We represent a 3-coloring by the pair $(c, m)$, where $c \in \{0, 1, 2\}$ is the color of the center vertex and $m \in \{0, 1, \ldots, n-1\}$ is the number of leaves colored $c+1 \pmod{3}$. The remaining $n-1-m$ leaves are colored $c+2 \pmod{3}$. Let $\ell(\cdot)$ denote the length of a path in the 3-coloring graph. In the paths below, center flips are performed only when all leaves avoid the new center color, and leaf recolorings are single-vertex steps.

Consider two colorings $x_1 = (c_1, m_1)$ and $x_2 = (c_2, m_2)$. If $c_1 = c_2$, then $\dist_{\C_3(T)}(x_1, x_2) = |m_1 - m_2| \le n-1 < \lfloor 3n/2 \rfloor$ for $n \ge 3$.

Now suppose $c_1 \neq c_2$. By symmetry, we may assume $c_1 = 0$. We consider the two cases $c_2 = 1$ and $c_2 = 2$ separately.

\textbf{Case 1:} $c_2 = 1$. We exhibit two paths from $(0, m_1)$ to $(1, m_2)$:
\begin{align*}
P_1 &\colon (0, m_1) \xrightarrow{m_1} (0, 0) \xrightarrow{1} (1, n-1) \xrightarrow{n-1-m_2} (1, m_2), \\
P_2 &\colon (0, m_1) \xrightarrow{n-1-m_1} (0, n-1) \xrightarrow{1} (2, 0) \xrightarrow{n-1} (2, n-1) \xrightarrow{1} (1, 0) \xrightarrow{m_2} (1, m_2).
\end{align*}
The path lengths are $\ell(P_1) = n + m_1 - m_2$ and $\ell(P_2) = 2n - m_1 + m_2$, so $\ell(P_1) + \ell(P_2) = 3n$. Thus, $\min(\ell(P_1), \ell(P_2)) \le \lfloor 3n/2 \rfloor$.

\textbf{Case 2:} $c_2 = 2$. We exhibit two paths from $(0, m_1)$ to $(2, m_2)$:
\begin{align*}
Q_1 &\colon (0, m_1) \xrightarrow{m_1} (0, 0) \xrightarrow{1} (1, n-1) \xrightarrow{n-1} (1, 0) \xrightarrow{1} (2, n-1) \xrightarrow{n-1-m_2} (2, m_2), \\
Q_2 &\colon (0, m_1) \xrightarrow{n-1-m_1} (0, n-1) \xrightarrow{1} (2, 0) \xrightarrow{m_2} (2, m_2).
\end{align*}
The path lengths are $\ell(Q_1) = 2n + m_1 - m_2$ and $\ell(Q_2) = n - m_1 + m_2$, so $\ell(Q_1) + \ell(Q_2) = 3n$. Thus, $\min(\ell(Q_1), \ell(Q_2)) \le \lfloor 3n/2 \rfloor$.

In all cases, $\dist_{\C_3(T)}(x_1, x_2) \le \lfloor 3n/2 \rfloor$, completing the proof.
\end{proof}

We note that Proposition~\ref{prop:star} appears in the literature (see \cite[Proposition 2]{CvBC24}). However, we include our proof above because it motivates the argument for the next result, which is new.

\begin{prop}\label{prop:double-star}
Suppose $a,b\in\mathbb{N}$ with $|a-b|\le 4$, and let $n=a+b+2$. Then the 3-coloring diameter of the double star $S(a,b)$ is $\lfloor 3n/2\rfloor$.
\end{prop}

\begin{proof}
The lower bound follows from Lemma~\ref{lem:universal-lower-bound}. For the upper bound, it suffices to show that any two 3-colorings $x_1$ and $x_2$ are at distance at most $\lfloor 3n/2\rfloor$.

\textbf{Notation.}
We encode a 3-coloring by $(c_1,c_2,m_1,m_2)$, where $c_i\in\{0,1,2\}$ are the colors of the two centers $v$ and $w$, and $m_1\in\{0,\ldots,a\}$ (resp.\ $m_2\in\{0,\ldots,b\}$) is the number of leaves at $v$ (resp.\ $w$) colored $c_1+1\pmod{3}$ (resp.\ $c_2+1\pmod{3}$). The remaining leaves at each center are colored $c_i+2\pmod{3}$. A center flip to a color $\gamma$ is valid if and only if every incident leaf avoids $\gamma$. Let $\ell(\cdot)$ denote path length in $\C_3(S(a,b))$.

\textbf{Same-center case.}
If two colorings $x_1, x_2$ have the same pair $(c_1,c_2)$, then recoloring only leaves yields
\[
\dist_{\C_3(T)}(x_1,x_2)\le a+b=n-2<\lfloor 3n/2\rfloor.
\]

\textbf{Reduction for distinct-center pairs.}
For the remaining pairs, at least one center color changes. Any permutation of $\{0,1,2\}$ induces an automorphism of $\C_3(S(a,b))$ and hence preserves distances. Thus, without loss of generality, we may relabel colors so that the two \emph{source} centers have colors $(c_1,c_2)=(0,1)$. After this normalization, the possible \emph{targets} fall, up to symmetry, into exactly three patterns:
\begin{align*}
&(0,1)\to(0,2)\quad\text{(single-center change)} \\ 
& (0,1)\to(1,2),\ (0,1)\to(1,0)\quad\text{(two-center changes)}.
\end{align*}

\textbf{Single-center change: $(0,1)\to(0,2)$.}
For $x_1=(0,1,m_1,m_2)$ and $x_2=(0,2,m'_1,m'_2)$ consider
\[
(0,1,m_1,m_2)\xrightarrow{m_2}(0,1,m_1,0)\xrightarrow{1}(0,2,m_1,b)\xrightarrow{|m_1-m'_1|}(0,2,m'_1,b)\xrightarrow{b-m'_2}(0,2,m'_1,m'_2),
\]
so $\ell\le m_2+1+|m_1-m'_1|+(b-m'_2)\le a+2b+1$. Therefore, to prove $a+2b+1 \le \lfloor \frac{3n}{2}\rfloor = \lfloor\frac{3(a+b+2)}{2}\rfloor$, it suffices to show that $b-2 \le \Big\lfloor \frac{a+b}{2}\Big\rfloor$, or equivalently, $b-2 \le \frac{a+b}{2}$. The latter inequality $2b-4\leq a+b$ holds because $b-a\le 4$.

\textbf{Two-center changes.}
It suffices to treat $(0,1)\to(1,2)$ and $(0,1)\to(1,0)$.

\textbf{Pattern $(0,1)\to(1,2)$.} We construct two paths in $\C_3(S(a,b))$ between $x_1$ and $x_2$:
\begin{align*}
R_1&\colon (0,1,m_1,m_2)\xrightarrow{m_1+m_2}(0,1,0,0)\xrightarrow{1}(0,2,0,b)\xrightarrow{1}(1,2,a,b)\xrightarrow{(a-m'_1)+(b-m'_2)}(1,2,m'_1,m'_2),\\
R_2&\colon (0,1,m_1,m_2)\xrightarrow{(a-m_1)+(b-m_2)}(0,1,a,b)\xrightarrow{2}(2,0,0,0)\\
&\hspace{3.3cm}\xrightarrow{a+b}(2,0,a,b)\xrightarrow{2}(1,2,0,0)\xrightarrow{m'_1+m'_2}(1,2,m'_1,m'_2),
\end{align*}
with $\ell(R_1)+\ell(R_2)=3a+3b+6=3n$, hence $\min\{\ell(R_1),\ell(R_2)\}\le\lfloor 3n/2\rfloor$.

\textbf{Pattern $(0,1)\to(1,0)$.}  We construct two paths in $\C_3(S(a,b))$ between $x_1$ and $x_2$:
\begin{align*}
S_1&\colon (0,1,m_1,m_2)\xrightarrow{(a-m_1)+(b-m_2)}(0,1,a,b)\xrightarrow{2}(2,0,0,0)\\
&\hspace{3.3cm}\xrightarrow{a}(2,0,a,0)\xrightarrow{1}(1,0,0,0)\xrightarrow{m'_1+m'_2}(1,0,m'_1,m'_2),\\
S_2&\colon (0,1,m_1,m_2)\xrightarrow{m_1+m_2}(0,1,0,0)\xrightarrow{2}(1,2,a,b)\\
&\hspace{3.3cm}\xrightarrow{b}(1,2,a,0)\xrightarrow{1}(1,0,a,b)\xrightarrow{(a-m'_1)+(b-m'_2)}(1,0,m'_1,m'_2),
\end{align*}
and again $\ell(S_1)+\ell(S_2)=3n$, so $\min\{\ell(S_1),\ell(S_2)\}\le\lfloor 3n/2\rfloor$.

Combining all the cases above, we conclude that $\dist_{\C_3(T)}(x_1,x_2)\le \lfloor 3n/2\rfloor$ for all pairs of colorings $x_1, x_2$. This proves $\diam(\C_3(S(a,b)))=\lfloor 3n/2\rfloor$.
\end{proof}

Together with Lemma~\ref{lem:universal-lower-bound}, this completes the proof of the second part of Theorem~\ref{thm:main}.

\appendix 

\section{Small trees and balanced labelings}\label{sec:small-trees}

In this appendix we list, up to isomorphism, all trees $T$ on $n\le 6$ vertices. For each tree, we display a balanced labeling $h\colon V\to \mathbb{Z}$ whose norm $\|h\|_1$ achieves $\diam(\C_3(T))$. 

Vertices are drawn as black dots, and the integer next to a vertex $v$ is the label $h(v)$. The bold number (in red font) beneath each figure is the 3-coloring diameter.

\medskip

\tikzset{
  smvtx/.style={circle,fill=black,inner sep=1.3pt},
  smedge/.style={line width=0.9pt},
  lab/.style={font=\scriptsize, inner sep=1pt}
}

\noindent\textbf{Trees on 1, 2, and 3 vertices}

\begin{center}
\begin{minipage}{0.28\textwidth}
\centering
\begin{tikzpicture}[scale=1]
  \node[smvtx] (v0) at (0,0) {};
  \node[lab, above=3pt of v0] {$1$};
\end{tikzpicture}

\vspace{2pt}
{\small \textcolor{red}{\textbf{1}}}
\end{minipage}
\hfill
\begin{minipage}{0.28\textwidth}
\centering
\begin{tikzpicture}[scale=1]
  \node[smvtx] (v0) at (0,0) {};
  \node[smvtx] (v1) at (1.2,0) {};
  \draw[smedge] (v0)--(v1);
  \node[lab, above=3pt of v0] {$1$};
  \node[lab, above=3pt of v1] {$2$};
\end{tikzpicture}

\vspace{2pt}
{\small \textcolor{red}{\textbf{3}}}
\end{minipage}
\hfill
\begin{minipage}{0.28\textwidth}
\centering
\begin{tikzpicture}[scale=1]
  \node[smvtx] (v0) at (0,0) {};
  \node[smvtx] (v1) at (1.0,0) {};
  \node[smvtx] (v2) at (2.0,0) {};
  \draw[smedge] (v0)--(v1);
  \draw[smedge] (v1)--(v2);
  \node[lab, above=3pt of v0] {$1$};
  \node[lab, above=3pt of v1] {$1$};
  \node[lab, above=3pt of v2] {$2$};
\end{tikzpicture}

\vspace{2pt}
{\small \textcolor{red}{\textbf{4}}}
\end{minipage}
\end{center}

\medskip 

\noindent\textbf{Trees on 4 vertices}

\begin{center}
\begin{minipage}{0.4\textwidth}
\centering
\begin{tikzpicture}[scale=1]
  \node[smvtx] (v0) at (0,0) {};
  \node[smvtx] (v1) at (1.0,0) {};
  \node[smvtx] (v2) at (2.0,0) {};
  \node[smvtx] (v3) at (3.0,0) {};
  \draw[smedge] (v0)--(v1);
  \draw[smedge] (v1)--(v2);
  \draw[smedge] (v2)--(v3);
  \node[lab, above=3pt of v0] {$1$};
  \node[lab, above=3pt of v1] {$1$};
  \node[lab, above=3pt of v2] {$2$};
  \node[lab, above=3pt of v3] {$2$};
\end{tikzpicture}

\vspace{2pt}
{\small \textcolor{red}{\textbf{6}}}
\end{minipage}
\hfill
\begin{minipage}{0.4\textwidth}
\centering
\begin{tikzpicture}[scale=1]
  \node[smvtx] (c)  at (0,0) {};
  \node[smvtx] (v1) at (1.2,0) {};
  \node[smvtx] (v2) at (-0.9,0.9) {};
  \node[smvtx] (v3) at (-0.9,-0.9) {};
  \draw[smedge] (c)--(v1);
  \draw[smedge] (c)--(v2);
  \draw[smedge] (c)--(v3);
  \node[lab, above=3pt of c] {$1$};
  \node[lab, above=3pt of v1] {$1$};
  \node[lab, above=3pt of v2] {$2$};
  \node[lab, below=3pt of v3] {$2$};
\end{tikzpicture}

\vspace{2pt}
{\small \textcolor{red}{\textbf{6}}}
\end{minipage}
\end{center}

\medskip 

\noindent\textbf{Trees on 5 vertices}

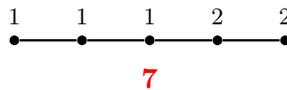
\begin{center}
\begin{minipage}{0.32\textwidth}
\centering
\begin{tikzpicture}[scale=1]
  \node[smvtx] (v0) at (0,0) {};
  \node[smvtx] (v1) at (0.9,0) {};
  \node[smvtx] (v2) at (1.8,0) {};
  \node[smvtx] (v3) at (2.7,0) {};
  \node[smvtx] (v4) at (3.6,0) {};
  \draw[smedge] (v0)--(v1);
  \draw[smedge] (v1)--(v2);
  \draw[smedge] (v2)--(v3);
  \draw[smedge] (v3)--(v4);
  \node[lab, above=3pt of v0] {$1$};
  \node[lab, above=3pt of v1] {$1$};
  \node[lab, above=3pt of v2] {$1$};
  \node[lab, above=3pt of v3] {$2$};
  \node[lab, above=3pt of v4] {$2$};
\end{tikzpicture}

\vspace{2pt}
{\small \textcolor{red}{\textbf{7}}}
\end{minipage}
\hfill
\begin{minipage}{0.32\textwidth}
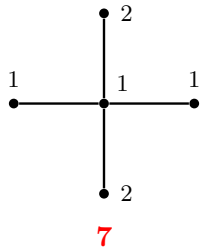
\centering
\begin{tikzpicture}[scale=1]
  \node[smvtx] (c)  at (0,0) {};
  \node[smvtx] (v1) at (1.2,0) {};
  \node[smvtx] (v2) at (-1.2,0) {};
  \node[smvtx] (v3) at (0,1.2) {};
  \node[smvtx] (v4) at (0,-1.2) {};
  \draw[smedge] (c)--(v1);
  \draw[smedge] (c)--(v2);
  \draw[smedge] (c)--(v3);
  \draw[smedge] (c)--(v4);
  \node[lab, above right=3pt of c] {$1$};
  \node[lab, above=3pt of v1] {$1$};
  \node[lab, above=3pt of v2] {$1$};
  \node[lab, right=3pt of v3] {$2$};
  \node[lab, right=3pt of v4] {$2$};
\end{tikzpicture}

\vspace{2pt}
{\small \textcolor{red}{\textbf{7}}}
\end{minipage}
\hfill
\begin{minipage}{0.32\textwidth}
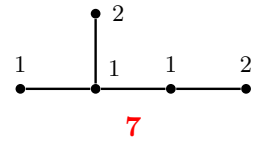
\centering
\begin{tikzpicture}[scale=1]
  \node[smvtx] (v1) at (0,0) {};  
  \node[smvtx] (v0) at (-1,0) {};  
  \node[smvtx] (v2) at (1,0) {};  
  \node[smvtx] (v3) at (0,1.0) {}; 
  \node[smvtx] (v4) at (2,0) {}; 
  \draw[smedge] (v1)--(v0);
  \draw[smedge] (v1)--(v2);
  \draw[smedge] (v1)--(v3);
  \draw[smedge] (v2)--(v4);
  \node[lab, above=3pt of v0] {$1$};
  \node[lab, above right=3pt of v1] {$1$};
  \node[lab, above=3pt of v2] {$1$};
  \node[lab, right=3pt of v3] {$2$};
  \node[lab, above=3pt of v4] {$2$};
\end{tikzpicture}

\vspace{2pt}
{\small \textcolor{red}{\textbf{7}}}
\end{minipage}
\end{center}

\medskip 

\noindent\textbf{Trees on 6 vertices}

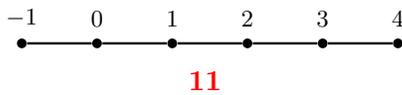
\begin{center}
\begin{minipage}{0.48\textwidth}
\centering
\begin{tikzpicture}[scale=1]
  \node[smvtx] (v0) at (-2.5,0) {};
  \node[smvtx] (v1) at (-1.5,0) {};
  \node[smvtx] (v2) at (-0.5,0) {};
  \node[smvtx] (v3) at (0.5,0) {};
  \node[smvtx] (v4) at (1.5,0) {};
  \node[smvtx] (v5) at (2.5,0) {};
  \draw[smedge] (v0)--(v1);
  \draw[smedge] (v1)--(v2);
  \draw[smedge] (v2)--(v3);
  \draw[smedge] (v3)--(v4);
  \draw[smedge] (v4)--(v5);
  \node[lab, above=3pt of v0] {$-1$};
  \node[lab, above=3pt of v1] {$0$};
  \node[lab, above=3pt of v2] {$1$};
  \node[lab, above=3pt of v3] {$2$};
  \node[lab, above=3pt of v4] {$3$};
  \node[lab, above=3pt of v5] {$4$};
\end{tikzpicture}

\vspace{2pt}
{\small \textcolor{red}{\textbf{11}}}
\end{minipage}%
\hfill
\begin{minipage}{0.48\textwidth}
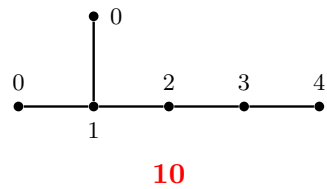
\centering
\begin{tikzpicture}[scale=1]
  \node[smvtx] (w0) at (-2,0) {}; 
  \node[smvtx] (w1) at (-1,0) {};  
  \node[smvtx] (w2) at (0,0) {};  
  \node[smvtx] (w3) at (1,0) {};   
  \node[smvtx] (w4) at (2,0) {};  
  \node[smvtx] (w5) at (-1,1.2) {}; 

  \draw[smedge] (w0)--(w1);
  \draw[smedge] (w1)--(w2);
  \draw[smedge] (w2)--(w3);
  \draw[smedge] (w3)--(w4);
  \draw[smedge] (w1)--(w5);

  \node[lab, above=3pt of w0] {$0$};
  \node[lab, below=3pt of w1] {$1$};
  \node[lab, above=3pt of w2] {$2$};
  \node[lab, above=3pt of w3] {$3$};
  \node[lab, above=3pt of w4] {$4$};
  \node[lab, right=3pt of w5] {$0$};
\end{tikzpicture}

\vspace{2pt}
{\small \textcolor{red}{\textbf{10}}}
\end{minipage}

\medskip

\begin{minipage}{0.48\textwidth}
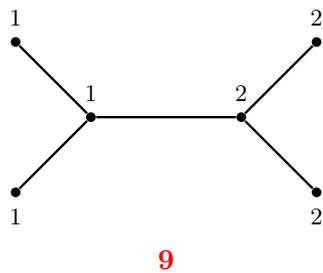
\centering
\begin{tikzpicture}[scale=1]
  \node[smvtx] (c1) at (-1,0) {};     
  \node[smvtx] (c2) at (1,0) {};     
  \node[smvtx] (a1) at (-2,1) {};
  \node[smvtx] (a2) at (-2,-1) {};
  \node[smvtx] (b1) at (2,1) {};
  \node[smvtx] (b2) at (2,-1) {};
  \draw[smedge] (c1)--(a1);
  \draw[smedge] (c1)--(a2);
  \draw[smedge] (c1)--(c2);
  \draw[smedge] (c2)--(b1);
  \draw[smedge] (c2)--(b2);
  \node[lab, above=3pt of c1] {$1$};
  \node[lab, above=3pt of c2] {$2$};
  \node[lab, above=3pt of a1] {$1$};
  \node[lab, below=3pt of a2] {$1$};
  \node[lab, above=3pt of b1] {$2$};
  \node[lab, below=3pt of b2] {$2$};
\end{tikzpicture}

\vspace{2pt}
{\small \textcolor{red}{\textbf{9}}}
\end{minipage}%
\hfill
\begin{minipage}{0.48\textwidth}
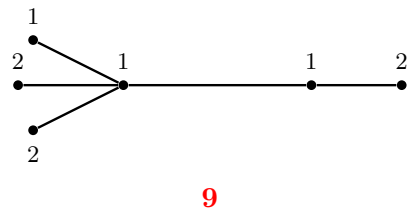
\centering
\begin{tikzpicture}[scale=1]

  \node[smvtx] (d1) at (-1,0) {};   
  \node[smvtx] (d2) at (1.5,0) {};  
  \node[smvtx] (e1) at (-2.2,0.6) {};   
  \node[smvtx] (e2) at (-2.4,0.0) {};   
  \node[smvtx] (e3) at (-2.2,-0.6) {};  
  \node[smvtx] (f1) at (2.7,0) {};
  \draw[smedge] (d1)--(d2);
  \draw[smedge] (d1)--(e1);
  \draw[smedge] (d1)--(e2);
  \draw[smedge] (d1)--(e3);
  \draw[smedge] (d2)--(f1);
  \node[lab, above=3pt of d1] {$1$};
  \node[lab, above=3pt of d2] {$1$};

  \node[lab, above=3pt of e1] {$1$};
  \node[lab, above=3pt of e2] {$2$};
  \node[lab, below=3pt of e3] {$2$};

  \node[lab, above=3pt of f1] {$2$};
\end{tikzpicture}

\vspace{2pt}
{\small \textcolor{red}{\textbf{9}}}
\end{minipage}

\medskip

\begin{minipage}{0.48\textwidth}
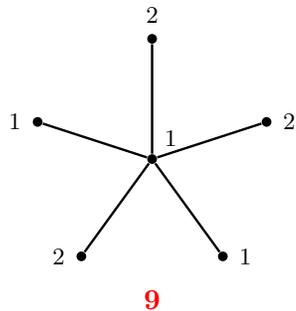
\centering
\begin{tikzpicture}[scale=1]
  \node[smvtx] (c) at (0,0) {};
  \foreach \i/\ang in {1/90,2/162,3/234,4/306,5/18} {
    \node[smvtx] (v\i) at ({1.6*cos(\ang)},{1.6*sin(\ang)}) {};
    \draw[smedge] (c)--(v\i);
  }
  \node[lab, above right=3pt of c] {$1$};
  \node[lab, above=3pt of v1] {$2$};
  \node[lab, left=3pt of v2] {$1$};
  \node[lab, left=3pt of v3] {$2$};
  \node[lab, right=3pt of v4] {$1$};
  \node[lab, right=3pt of v5] {$2$};
\end{tikzpicture}

\vspace{2pt}
{\small \textcolor{red}{\textbf{9}}}
\end{minipage}%
\hfill
\begin{minipage}{0.48\textwidth}
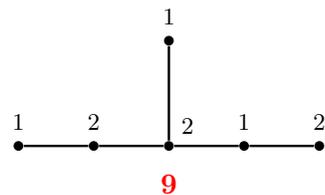
\centering
\begin{tikzpicture}[scale=1]
  \node[smvtx] (cB)  at (0,0) {};   
  \node[smvtx] (uL) at (-1,0) {};   
  \node[smvtx] (uR) at (1,0) {};  
  \node[smvtx] (uT) at (0,1.4) {};  
  \node[smvtx] (vL) at (-2,0) {};  
  \node[smvtx] (vR) at (2,0) {};  
  \draw[smedge] (cB)--(uL);
  \draw[smedge] (cB)--(uR);
  \draw[smedge] (cB)--(uT);
  \draw[smedge] (uL)--(vL);
  \draw[smedge] (uR)--(vR);
  \node[lab, above right=3pt of cB]  {$2$};
  \node[lab, above=3pt of uL] {$2$};
  \node[lab, above=3pt of uR] {$1$};
  \node[lab, above=3pt of uT] {$1$};
  \node[lab, above=3pt of vL] {$1$};
  \node[lab, above=3pt of vR] {$2$};
\end{tikzpicture}

\vspace{2pt}
{\small \textcolor{red}{\textbf{9}}}
\end{minipage}
\end{center}

For each $n \le 5$, all trees on $n$ vertices have the same 3-coloring diameter. Among all trees on $6$ vertices, the path $P_6$ is still the unique tree with maximum 3-coloring diameter, namely $11$. Besides the star $K_{1,5}$ and the double stars $S(2,2)$ and $S(3,1)$, there is another tree with degree sequence $3,2,2,1,1,1$ that also attains the minimum 3-coloring diameter, namely $9$.

\section*{Statements and Declarations}

The authors declare that no funds, grants, or other support were received during the preparation of this manuscript. Furthermore, the authors have no relevant financial or non-financial interests to disclose. 
All authors contributed equally to the research and preparation of the manuscript.

\bibliographystyle{alpha}
\bibliography{main}

\end{document}